\documentclass{amsart}
\usepackage{amsmath}
\usepackage{amssymb}
\usepackage{amsthm}
\usepackage{graphicx}
\usepackage{bm} 
\usepackage{mathtools}

\newcommand{\vex}[1]{#1_{\text{vex}}} 
\newcommand{\ave}[1]{#1_{\text{ave}}} 

\newcommand{\inhyperplane}{in-hyperplane}
\newcommand{\outhyperplane}{out-hyperplane}

\newtheorem{theorem}{Theorem}
\newtheorem{lemma}{Lemma}

\newtheorem{corollary}{Corollary}
\theoremstyle{definition}
\newtheorem{definition}{Definition}
\newtheorem{example}{Example}
\theoremstyle{remark}
\newtheorem{remark}{Remark}

\usepackage{mathrsfs}
\usepackage{subcaption}

\allowdisplaybreaks

\begin{document}
\title{Convex envelopes for ray-concave functions}
\thanks{The research leading to these results received funding from grants ANID/CONICYT-Fondecyt Regular 1200809 (J.B., E.M.) and ANID/CONICYT-Fondecyt Iniciaci\'on 11190515 (G.M.).}
\author{Javiera Barrera}
\address{Faculty of Engineering and Sciences,  Universidad Adolfo Ib\'a\~nez, Santiago, Chile}
\email{javiera.barrera@uai.cl}
\author{Eduardo Moreno}
\address{Faculty of Engineering and Sciences,  Universidad Adolfo Ib\'a\~nez, Santiago, Chile}
\email{eduardo.moreno@uai.cl} 
\author{Gonzalo Mu\~noz}
\address{Institute of Engineering Sciences, Universidad de O'Higgins,  Rancagua, Chile}
\email{gonzalo.munoz@uoh.cl}

\begin{abstract}
Convexification based on convex envelopes is ubiquitous 
in the non-linear optimization literature. Thanks to considerable efforts of the 
optimization community for decades, we are able to compute the convex envelopes of a considerable number of functions that appear in practice, and thus obtain tight and tractable approximations to challenging problems. We contribute to this line of work by considering a family of functions that, to the best of our knowledge, has not been considered before in the literature. We call this family \emph{ray-concave} functions. We show sufficient conditions that allow us to easily compute closed-form expressions for the convex envelope of ray-concave functions over arbitrary polytopes. With these tools, we are able to provide new perspectives to previously known convex envelopes and derive a previously unknown convex envelope for a function that arises in probability contexts.
\end{abstract}
\maketitle

\section{Introduction}

Strong convex relaxations of complex optimization problems is
a key component in the development of tractable computational techniques in the field.
In this regard, a popular approach has been the study of \emph{convex underestimators} of functions, that is, given an arbitrary function $f$, find a convex function $f'$ such that $f'(x) \leq f(x)$ $\forall x\in P$, where $P$ is a given convex set.
Such function can be used to \emph{relax} a sub-level set $\{x\in P :\, f(x) \leq 0 \}$ with the convex set $\{x\in P\, :\, f'(x)\leq 0\}$, and thus obtain a computationally tractable approximation.
The pointwise largest convex underestimator is known as the \emph{convex envelope of $f$ over $P$}, and the optimization community has allocated considerable efforts on finding such envelopes for various classes of functions $f$ and sets $P$.

\begin{definition} The convex envelope of a function $f$ on a subset $P$  is given by 
\begin{align*}
\vex{f}(x) &= \sup\{ g(x) \, :\,  g \text{ is convex and } g(x) \leq f(x) \ \forall x\in P\} 
\end{align*}
\end{definition}

In this work, we consider $P$ a polytope and study the convex envelope of a family of functions that are required to be convex of the facets of $P$, and what we term as \emph{ray-concavity}.
\begin{definition}\label{def:rayconcave}
A function $f:P\to \mathbb{R}$ is \emph{ray-concave} over $P$ if, for every $x\in P$, the function $f$ restricted to $\{\alpha x \, :\, \alpha \geq 0\} \cap P$ is concave.
\end{definition}
We present sufficient conditions for deriving simple closed-form formulas of the convex envelopes of ray-concave functions over arbitrary polytopes in any dimension.
%

Our result is closely related to known results for general functions over polytopes. To the best of our knowledge, the vast majority of the work producing closed-form formulas of convex envelopes in arbitrary dimension either require a rectangular domain, or require $f$ to be \emph{edge-concave}, in which case the convex envelope is polyhedral\footnote{A function is \emph{polyhedral} if its epigraph is a polyhedron.}.
With our result, through the concept of ray-concavity, we are able to explicitly construct convex envelopes which are not necessarily polyhedral, in any dimension, for a new family of functions that has not been explicitly exploited before in the literature.
%

Our result yields a previously unknown convex envelope of a function that appears in probability contexts.
\begin{example}\label{ex:probability_function}
The function $f(x,y) = -\frac{x\cdot y}{x+y-x\cdot y}$ is ray-concave over any box $[0,u_x]\times [0,u_y]$ with  $u_x,u_y\leq 1$.
\end{example}
This function is one of the the main motivations behind this work.
Additionally, many functions for which their convex envelope formulas are known exhibit ray-concavity (e.g., $f(x_1,x_2)=-x_1 x_2$ or $f(x_1,x_2)=x_1/x_2$ for $x_1,x_2>0$), and our result provide a new perspective on these expressions and alternative derivations.

\section{Literature review}

The literature of convex envelopes is vast. Probably the most well-known and used convex envelope is that of the bilinear function $f(x_1,x_2) = x_1x_2$ over a rectangular region, for which its convex (and concave) envelope is obtained through the McCormick envelopes~\cite{mccormick1976computability,al1983jointly}. 

To the best of our knowledge, the first method capable of constructing the convex envelope for a family of functions (as opposed to a particular function) is provided in \cite{tawarmalani2001semidefinite}. Based on disjunctive programming, they show a general expression of the convex envelope for functions that are concave on one variable, convex on the rest, and defined over a rectangular region. Later on, in~\cite{jach2008convex} the authors show how to compute the evaluation of $\vex{f}$ when $f$ is an $(n-1)$-convex function (i.e., $f$ is convex whenever one variable is fixed to any value) over a rectangular domain. The function evaluation requires the resolution of a convex optimization problem. 
In~\cite{khajavirad2012convex,khajavirad2013convex}, the authors formulate the convex envelope of a lower semi-continuous function over a compact set as a convex optimization problem. They use this to compute, explicitly, the convex envelope for various functions that are the product of a convex function and a component-wise concave function, over a box.  
We remark that in all the aforementioned cases, the convex envelopes may be non-polyhedral, and that explicit calculations consider hyper-rectangular domains.

Considerable efforts have been put in the case of polyhedral convex envelopes. In \cite{tardella2004existence,tardella2008existence}, it is shown that \emph{edge-concavity} of a function $f$ (i.e., concavity over all edge directions of a polytope $P$) implies that the convex envelope of $f$ over $P$ is polyhedral. The construction of these convex envelopes is studied in \cite{meyer2005convex}.
In \cite{rikun1997convex}, necessary and sufficient conditions for the convex envelope to be polyhedral are also provided, and they are used to obtain the convex envelope of a multilinear function over the unit box (see also \cite{sherali1997convex,ryoo2001analysis}).
In \cite{meyer2004trilinear}, the authors provide explicit expressions for the facets of the convex envelope of a trilinear monomial over a box.
In \cite{bao2009multiterm}, the authors design a cutting plane approach to generate, on-the-fly, the convex envelope of a bilinear function over a box. The strength of the convex underestimator of a bilinear function that is obtained from using a term-wise convex envelopes is analyzed in \cite{luedtke2012some}.

Other known results include the convex envelopes of odd-degree monomials over an interval \cite{liberti2003convex} and the fractional function $f(x_1,x_2)=x_1/x_2$ over a rectangle \cite{zamora1999branch,tawarmalani2001semidefinite}. Recently, in \cite{locatelli2020convex} the author computed the convex envelope of cubic functions in two dimensions, over a rectangular region.

While a big portion of these works involve rectangular regions, there exist important work considering sets beyond boxes in two dimensions. In \cite{sherali1990explicit}, the authors derive explicit formulas for the convex envelope of bilinear bivariate functions over a class special polytopes called $D$-polytopes. 
The case of the fractional function $x_1/x_2$ over a trapezoid is studied in \cite{kuno2002branch}. This was expanded in \cite{benson2004construction}, where convex envelopes for bilinear and fractional bivariate functions over quadrilaterals are constructed. 
The convex envelope of a bilinear bivariate function over a triangle has been carefully studied in \cite{sherali1990explicit,linderoth2005simplicial,anstreicher2010computable}. Such envelopes were tested computationally in \cite{linderoth2005simplicial} within a branching scheme for QCQPs with positive results. 
In \cite{locatelli2014convex} it is shown how to evaluate the convex envelope, and obtain a supporting hyperplane, for bivariate functions over arbitrary polytopes. This approach involves solving a low-dimensional convex problem. This procedure was refined in \cite{locatelli2018convex}, by shifting the calculations to the solution of a KKT system. These last techniques were extensively tested in \cite{muller2020using} to improve general-purpose optimization routines. In \cite{locatelli2016polyhedral}, the author characterizes the convex envelope of various bivariate functions (including the bilinear and fractional functions) over arbitrary polytopes using a polyhedral sub-division of the polytopes. In some cases, the convex envelope in each element of the sub-division can be given explicitly.

To the best of our knowledge, there is no construction that can provide a closed-form formula for the convex envelope of Example \ref{ex:probability_function}, and almost no construction allowing the explicit computation of non-polyhedral convex envelopes over polytopes beyond boxes in dimension $n\geq 3$. 
The only exception that we are aware of is \cite{tawarmalani2013explicit}. In this work, the authors derive explicit convex and concave envelopes of several functions on sub-sets of a hyper-rectangle, which are obtained through polyhedral subdivisions. In this case the authors can obtain, in closed form, the convex envelope of disjunctive functions of the form $xf(y)$, and the concave envelope of \emph{concave-extendable} supermodular functions. This may produce non-polyhedral envelopes. We remark some similarities with their construction below, however, it is worth noting that the results in \cite{tawarmalani2013explicit} cannot directly provide a formula for $\vex{f}$ for the function $f$ in Example \ref{ex:probability_function}.
On one hand, such function $f$ does not fit the disjunctive framework of \cite{tawarmalani2013explicit}, so we cannot apply their convex envelope construction. On the other hand, one could consider using their concave envelope results with $-f$, thus effectively constructing $\vex{f}$. However, we show below that the convex envelope of such $f$ requires a polyhedral partition that introduces new vertices in the box, while the construction of the concave envelope for concave-extendable functions (see \cite[Corollary 2.8]{tawarmalani2013explicit}) is based solely on the original vertices of the polytope.



\section{Convex envelopes for ray-concave functions}

Overall, we consider a polytope $P\subset \mathbb{R}^n$ with non-empty interior. 



\begin{definition} For any $v\in \mathbb{R}^n\setminus\{0\}$ such that $\exists\, \alpha \geq 0, \alpha v \in P$ (i.e., the ray defined by $v$ intersects the polytope) we define
\begin{align*}
v^+ &= \alpha^+ v \text{, where } \alpha^+ = \arg\max\{\alpha\, :\, \alpha \geq 0, \, \alpha v \in P \}& \text{ and }\\
v^- &= \alpha^- v \text{, where } \alpha^- = \arg\min\{\alpha\, :\, \alpha \geq 0, \, \alpha v \in P \}.&
\end{align*}
In simple words, $v^+$ and $v^-$ are the intersections of the ray given by $v$ with the boundary of $P$ (see Figure~\ref{fig:notation}). Note that if $0\in P$ then $v^-=0$ for all $v\in P$.
\end{definition}
We remark that $v^\pm$ are continuous functions of $v$. Below, we emphasize this functional aspect when taking derivatives.

Using this definition, a function $f:P \rightarrow \mathbb{R}$ is ray-concave iff $f$ restricted to the segment $[v^-,v^+]$ is concave for all $v$ where $v^{\pm}$ is well defined.
%
%
%
Our main results provides an explicit characterization for the convex envelope of ray-concave functions that are convex on the facets of $P$.

\begin{theorem}\label{thm:main}
 Let $f:P\to \mathbb{R}$ be a continuously differentiable and ray-concave function over a polytope $P$, such that $f$ is convex over the facets of $P$. 
 Let $g:P\to \mathbb{R}$ be defined as
 \begin{align}
	g(v) = \alpha_v f(v^-) + (1-\alpha_v) f(v^+), \label{eq:def}
\end{align}
where $\alpha_v\in[0,1]$ is such that $v = \alpha_v v^- + (1-\alpha_v) v^+$. If $g$ is positively homogeneous, then $\vex{f} = g$.
\end{theorem}
\begin{remark}\label{remark:first}
In Section \ref{sec:positivehom} we provide more insights on the positively homogeneous requirement. For example, we show that whenever $0\in P$,  $g$ is positively homogeneous iff $f(0) = 0$. The latter is not a restrictive requirement, as we can compute the convex envelope of $f - f(0)$ instead.
\end{remark}

A linear interpolation of a similar type as \eqref{eq:def} has been considered in multiple articles. The general result in \cite{jach2008convex}, for example, shows that to evaluate $\vex{f}$ for an edge-convex function $f$ over a box, it suffices to consider the lines passing through $x$ where the function $f$ is concave, similarly to our result. Each evaluation involves solving an optimization problem (see \cite[Theorem 3.1]{jach2008convex}). Another example is given by \cite{tawarmalani2013explicit}, who construct envelopes explicitly using secants of a similar type. In  \cite{linderoth2005simplicial}, the author also uses such lines in his construction of convex envelopes of the bilinear function over triangles.

In our case, by considering ray-concavity, we only need to consider secants on the rays emanating from the origin in the envelope construction.\\

To prove Theorem \ref{thm:main}, we first provide three lemmas about the convexity of the function $g$ over different regions of the domain. We divide the polytope $P$ into subregions using the rays that pass through the vertices of $P$. 

\begin{definition}\label{def:regions}
Let $\mathcal{F}$ be the set of facets of $P$. 
If $0\notin P$, for each pair of facets $F_i,F_j\in\mathcal{F}$ we define the region
\[ B_{ij} = \{v\in P : v^-\in F_i, v^+\in F_j \}.\]
We refer to the hyperplane containing the facet $F_i$ as the \emph{\inhyperplane{} of $B_{ij}$}, and to the hyperplane containing the facet $F_j$ as the \emph{\outhyperplane{} of $B_{ij}$}.
Alternatively, if $0\in P$, for each facet $F_j\in \mathcal{F}$ we define the region
\[ B_{0j} = \{v\in P : v^+\in F_j \}.\]
In this case we only define the \emph{\outhyperplane{}} of $B_{0j}$.
We denote by $\mathcal{B}$ the set of all full-dimensional regions $B_{ij}$. 
\end{definition}

In Figure \ref{fig:subdivisions} we illustrate the regions we consider in $\mathcal{B}$, which clearly form a sub-division of $P$.
Note that if $B_{ij}\neq \emptyset$ then $B_{ji} = \emptyset$.
\begin{figure}
    \centering
    \begin{subfigure}[t]{0.45\textwidth}
\includegraphics[height=5cm]{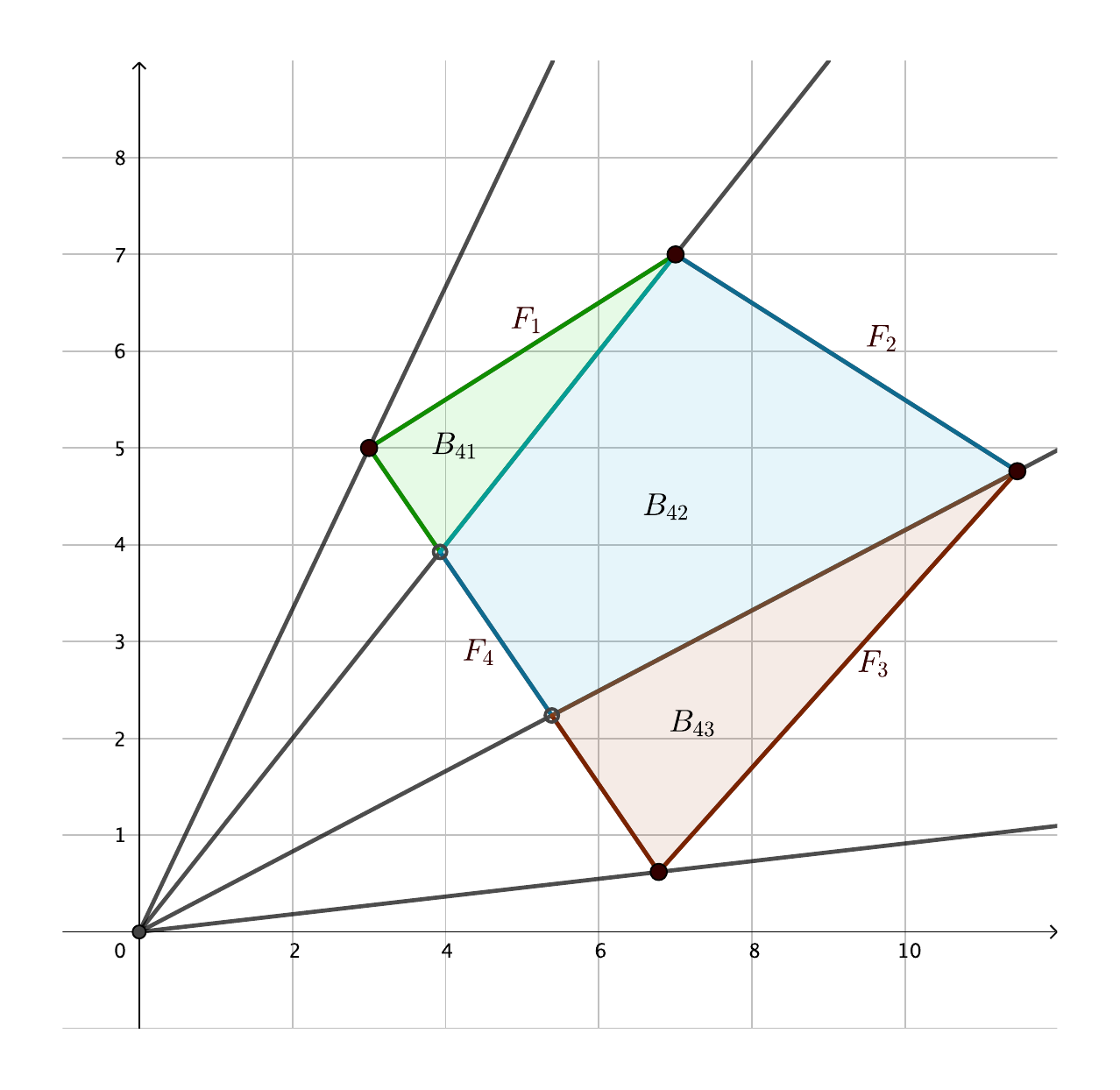}
\end{subfigure}
\begin{subfigure}[t]{0.45\textwidth}
\includegraphics[height=5cm]{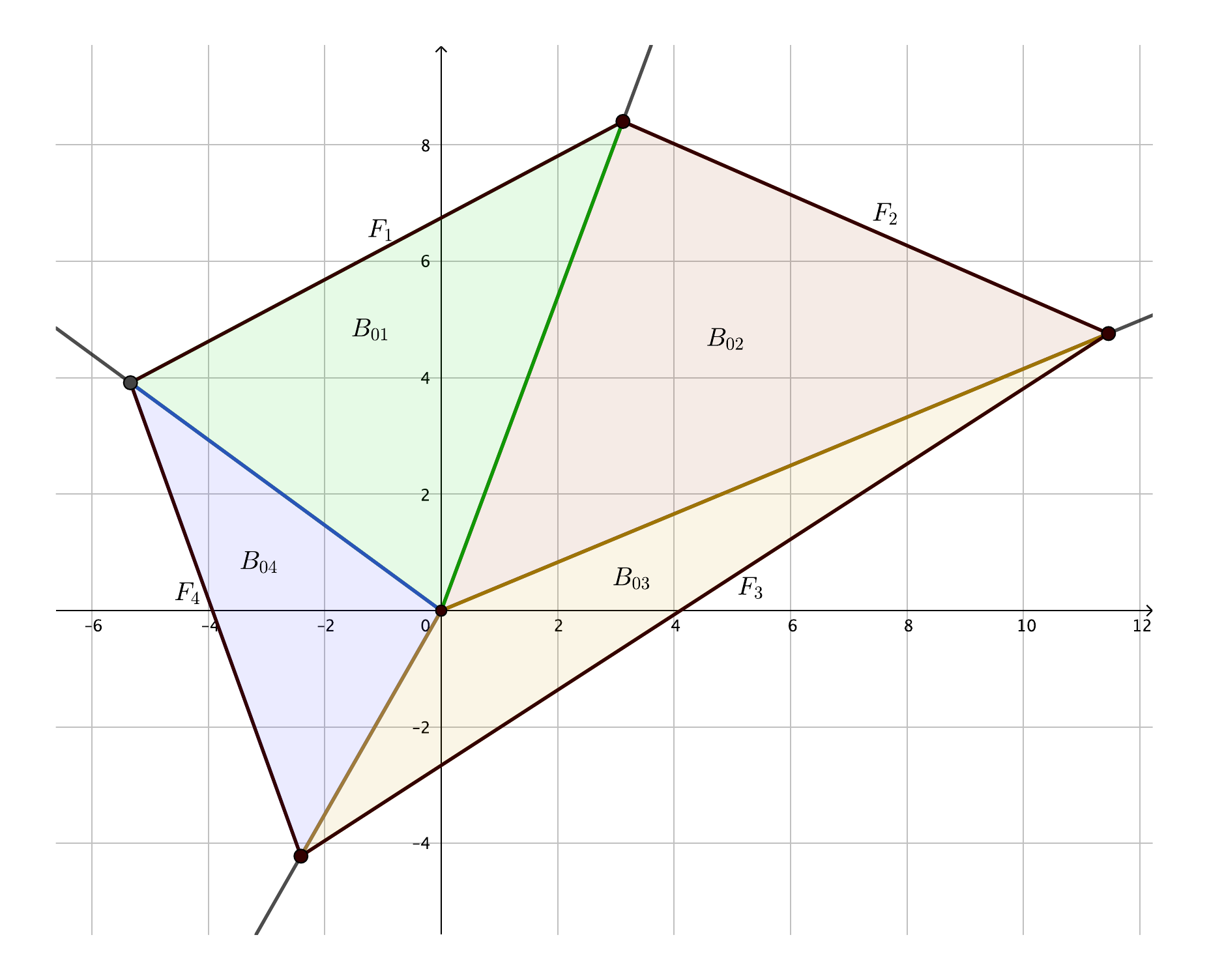}
\end{subfigure}
\caption{Polyhedral sub-division of $P$ into regions $\mathcal{B}$ according to intersection of rays with the boundary}
    \label{fig:subdivisions}
\end{figure}
Also note that every $B_{ij}\in \mathcal{B}$ is polyhedral: for example, in the case $0\not\in P$, it is not hard to see that
\begin{equation}
    \label{eq:bijascones}
    B_{ij} = \mbox{cone}(F_i) \cap \mbox{cone}(F_j) \cap P.
\end{equation}
Polyhedrality follows since both $F_i$ and $F_j$ are polyhedra. 

\begin{remark}\label{vpm}
For a given region $B_{ij}\in \mathcal{B}$ we can provide an explicit formula for $v^\pm$.
In fact, note that we can assume that the \outhyperplane{} of $B_{ij}$ has the form ${a^+}^\intercal \vec{x} =1$. Since $v$ and $v^+$ lie on the same ray, we obtain 
$v^+ = \frac{1}{{a^+}^\intercal v} v$
for any $v\in B_{ij}$.
Similarly, for the case $0\notin P$, we may assume that the \inhyperplane{} of $B_{ij}$ has the form ${a^-}^\intercal \vec{x} =1$, and then $v^- = \frac{1}{{a^-}^\intercal v} v$ for any $v\in B_{ij}$.

Moreover, since $v=\alpha_v v^- + (1-\alpha_v) v^+$, this implies that
\begin{equation}
    \frac{\alpha_v}{{a^-}^\intercal v} + \frac{1-\alpha_v}{{a^+}^\intercal v} = 1 \label{eq:relAlphaHiperplanes}
\end{equation}
\end{remark}

\subsection{Convexity and differentiability over a single region}

To show convexity of $g$, we first prove that under the homogeneity assumption of Theorem \ref{thm:main}, $g$ is convex in each region $B\in\mathcal{B}$.
%
%
\begin{lemma}
Let $B\in\mathcal{B}$ and let $g: B\subset \mathbb{R}^n\rightarrow \mathbb{R}$ as defined in \eqref{eq:def}.
%
%
If $g$ is positively homogeneous, then $g$ is convex in $B$.
\end{lemma}
\begin{proof}

%
%

Let $v,w\in B$ and let $z=\lambda v + (1-\lambda) w$ for $\lambda\in[0,1]$.  By convexity of the region, $z\in B$ as well. To prove the convexity of $g$ over $B$, we show that $g(z) \leq \lambda g(v) + (1-\lambda) g(w)$.

Recall that $v^+$ and $w^+$ belong to the same facet defining $B$, and $v^-$ and $w^-$ are either $0$ (if $0\in P$) or belong to the same facet defining $B$ (if $0\notin P$). Hence, there exist $\gamma, \varepsilon, \delta \in [0,1]$ such that:
\begin{align*}
    z&=\gamma z^- + (1-\gamma) z^+\\
    z^- &= \varepsilon v^- + (1-\varepsilon) w^-\\
    z^+ &= \rho v^+ + (1-\rho) w^+
\end{align*}
In Figure~\ref{fig:notation} we illustrate these vectors.

\begin{figure}[tbp]
\centering \includegraphics[width=.7\linewidth]{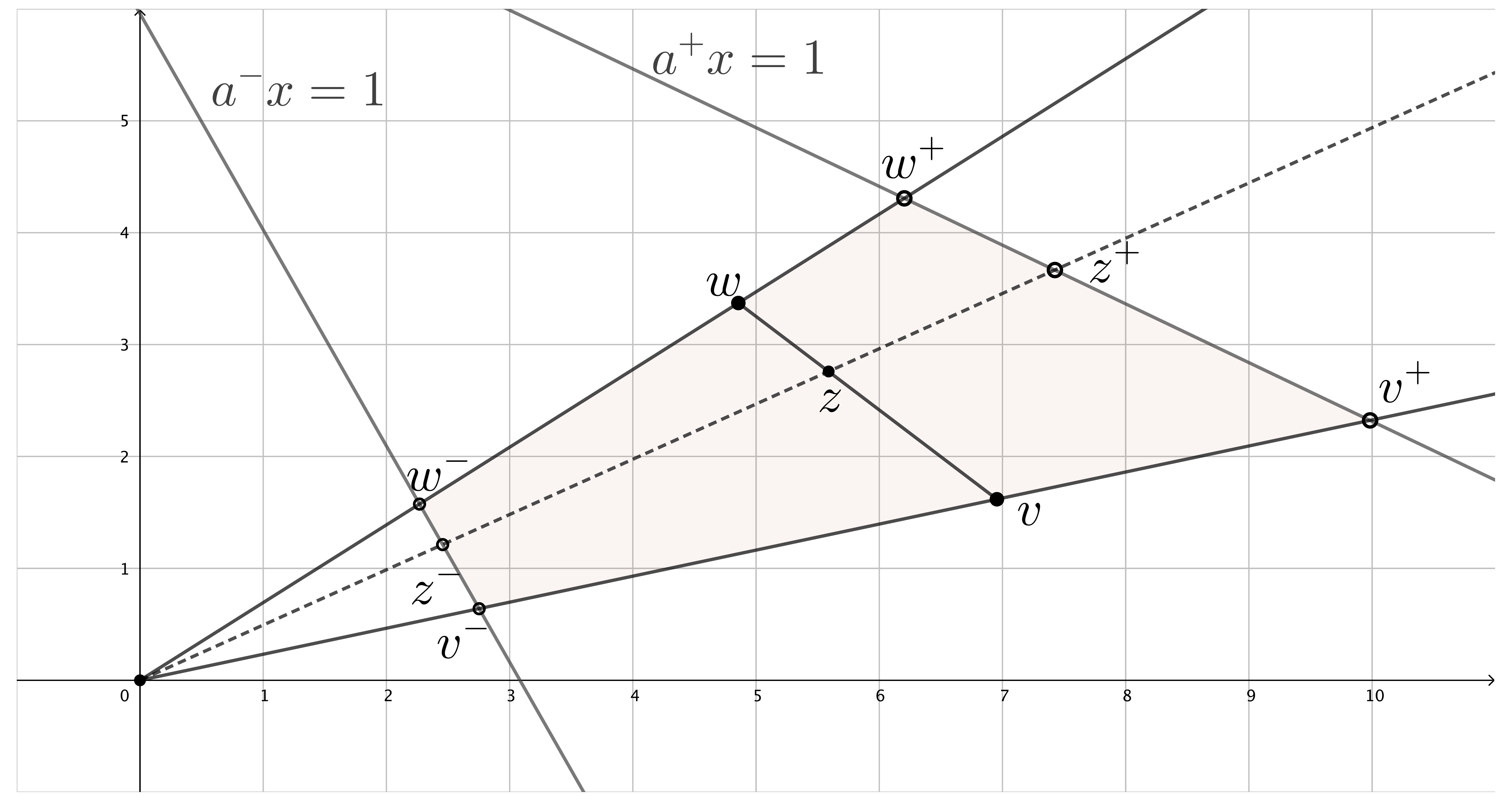}%
\caption{Notation for Lemma 1}\label{fig:notation}
\end{figure}

Since $g(z) = \gamma f(z^-) + (1-\gamma) f(z^+)$ and $f$ is convex on the facets containing $\{v^+,z^+,w^+\}$ and $\{v^-,z^-,w^-\}$ (if $0\notin P$), we know that
\begin{align}
g(z) &= \gamma f(z^-) + (1-\gamma) f(z^+) \\
& \leq \gamma (\varepsilon f(v^-) + (1-\varepsilon) f(w^-)) + (1-\gamma) (\rho f(v^+) + (1-\rho) f(w^+)) \\
& = \gamma \varepsilon f(v^-) +(1-\gamma) \rho f(v^+) + \gamma (1-\varepsilon) f(w^-) +   (1-\gamma)(1-\rho) f(w^+) \label{eq:despejada}
\end{align}


Let ${a^+}^\intercal \vec{x}=1$ be the \outhyperplane{} of $B$, i.e., the hyperplane that contains $v^+$, $w^+$, and $z^+$. By Remark~\ref{vpm}, we know that
$$z^+ = \frac{1}{{a^+}^\intercal z} z = \frac{1}{{a^+}^\intercal z}\left(\lambda v + (1-\lambda) w\right) =  \underbrace{\lambda \frac{{a^+}^\intercal v}{{a^+}^\intercal z}}_{\rho} v^+ + \underbrace{(1-\lambda)  \frac{{a^+}^\intercal w}{{a^+}^\intercal z}}_{1-\rho} w^+$$
where we deduce $ \rho= \lambda \frac{{a^+}^\intercal v}{{a^+}^\intercal z}$ because ${a^+}^\intercal z = \lambda {a^+}^\intercal v + (1-\lambda) {a^+}^\intercal w$. 
In a similar way, when $0\notin P$ we can apply the same for $z^-$ we obtain 
\begin{equation}
    \varepsilon =  \lambda \frac{{a^-}^\intercal v}{{a^-}^\intercal z}  \qquad\text{and} \qquad  1-\varepsilon =  (1-\lambda) \frac{{a^-}^\intercal w}{{a^-}^\intercal z} \label{eq:varepsilon}
\end{equation}
where ${a^-}^\intercal \vec{x}=1$ is the \inhyperplane{} of $B$.
If $0\in P$, then $v^-=w^-=z^-=0$ and thus $\varepsilon$ can take any value in $[0,1]$. To simplify the proof, we abuse notation and consider $\frac{{a^-}^\intercal v}{{a^-}^\intercal z}=\frac{{a^-}^\intercal w}{{a^-}^\intercal z}=1$ for this case, so \eqref{eq:varepsilon} still holds.

Replacing the values of $\rho$ and $\varepsilon$ in \eqref{eq:despejada}, we obtain
\begin{align}
    g(z) &\leq \gamma \varepsilon f(v^-) +(1-\gamma) \rho f(v^+) + \gamma (1-\varepsilon) f(w^-) +   (1-\gamma)(1-\rho) f(w^+) \nonumber\\
 &=  \begin{multlined}[t]\gamma \lambda\frac{{a^-}^\intercal v}{{a^-}^\intercal z}f(v^-) + (1-\gamma) \lambda \frac{{a^+}^\intercal v}{{a^+}^\intercal z} f(v^+) \\
 + \gamma (1-\lambda)\frac{{a^-}^\intercal w}{{a^-}^\intercal z} f(w^-) +   (1-\gamma)(1-\lambda)\frac{{a^+}^\intercal w}{{a^+}^\intercal z} f(w^+)\end{multlined} \nonumber \\
 &=  \begin{multlined}[t]\lambda \left( \gamma \frac{{a^-}^\intercal v}{{a^-}^\intercal z}f(v^-) + (1-\gamma) \frac{{a^+}^\intercal v}{{a^+}^\intercal z} f(v^+) \right) \\ 
\shoveright{+ (1-\lambda) \left(\gamma \frac{{a^-}^\intercal w}{{a^-}^\intercal z} f(w^-) +   (1-\gamma)\frac{{a^+}^\intercal w}{{a^+}^\intercal z} f(w^+) \right).} \end{multlined} \label{eq:lastinequalityg}
\end{align}
%
What follows uses that $g$ is positively homogeneous in order to rewrite the last inequality. To do so, note that 
\begin{align}
\gamma \frac{{a^-}^\intercal v}{{a^-}^\intercal z} \cdot v^- + (1-\gamma) \frac{{a^+}^\intercal v}{{a^+}^\intercal z} \cdot v^+ &= \gamma \frac{1}{{a^-}^\intercal z} v + (1-\gamma) \frac{1}{{a^+}^\intercal z} v  = v  \label{eq:rewritev}
\end{align}
because $z=\gamma z^- + (1-\gamma) z^+ = \left(\gamma \frac{1}{{a^-}^\intercal z} + (1-\gamma) \frac{1}{{a^+}^\intercal z}\right) z $.
Let $\Omega= \gamma \frac{{a^-}^\intercal v}{{a^-}^\intercal z} +  (1-\gamma) \frac{{a^+}^\intercal v}{{a^+}^\intercal z}$ ---this is simply the sum of the weights in the leftmost linear combination of \eqref{eq:rewritev}.
By definition of $g$, and because we are assuming it to be positively homogeneous, we have that 
\begin{align*}
 g(v) &= g\left( \gamma \frac{{a^-}^\intercal v}{{a^-}^\intercal z} \cdot v^- + (1-\gamma) \frac{{a^+}^\intercal v}{{a^+}^\intercal z} \cdot v^+ \right) && \text{(due to \eqref{eq:rewritev})}\\
&= \Omega \cdot g\left( \frac{\gamma \frac{{a^-}^\intercal v}{{a^-}^\intercal z}}{\Omega }  \cdot v^- + \frac{(1-\gamma) \frac{{a^+}^\intercal v}{{a^+}^\intercal z}}{\Omega } \cdot v^+ \right) && \text{(pos. homog.)} \\
&= \Omega \cdot \left( \frac{\gamma \frac{{a^-}^\intercal v}{{a^-}^\intercal z}}{\Omega }  \cdot f(v^-) + \frac{(1-\gamma) \frac{{a^+}^\intercal v}{{a^+}^\intercal z}}{\Omega } \cdot f(v^+) \right) && \text{(def. of $g$)} \\
&=  \gamma \frac{{a^-}^\intercal v}{{a^-}^\intercal z} \cdot f(v^-) + (1-\gamma) \frac{{a^+}^\intercal v}{{a^+}^\intercal z} \cdot f(v^+) 
\end{align*}
A similar relation can be deduced for $w$ obtaining
\[
\gamma \frac{{a^-}^\intercal w}{{a^-}^\intercal z} w^- + (1-\gamma)\frac{{a^+}^\intercal w}{{a^+}^\intercal z} w^+ =\left(\gamma \frac{1}{{a^-}^\intercal z} + (1-\gamma) \frac{1}{{a^+}^\intercal z}\right) w = w \]
which implies
\[ g(w) =g\left(\gamma \frac{{a^-}^\intercal w}{{a^-}^\intercal z} w^- +   (1-\gamma)\frac{{a^+}^\intercal w}{{a^+}^\intercal z} w^+\right) =  \gamma \frac{{a^-}^\intercal w}{{a^-}^\intercal z} f(w^-) +   (1-\gamma)\frac{{a^+}^\intercal w}{{a^+}^\intercal z} f(w^+). \]
Using these expressions for $g(v)$ and $g(w)$ in \eqref{eq:lastinequalityg} we obtain 
\[ g(z) = g\left(\lambda v + (1-\lambda) w\right) \leq \lambda g(v) + (1-\lambda) g(w). \]
This shows that $g$ is convex in $B$.
\end{proof}

The previous lemma shows that $g$ is convex in each region $B\in\mathcal{B}$. Before moving to convexity toward $P$, we show differentiability of $g$ in each region and compute the corresponding gradient, which we rely on in the next section.
\begin{lemma}\label{lemma:gradient}
Let $B\in \mathcal{B}$ and $g: B \subseteq \mathbb{R}^n_+\rightarrow \mathbb{R}$ as defined in \eqref{eq:def}. 
%
%
Let ${a^\pm}^\intercal x = 1$ be the \inhyperplane{} and \outhyperplane{} of $B$. Then, $g$ is a differentiable function in $\mbox{int}(B)$. Moreover, the gradient is given by
\begin{equation}
    \label{eq:gradientinside}
   \nabla g(v) = \frac{\alpha_v}{{a^-}^\intercal v} \left(\delta^- a^-  + \nabla f(v^-) \right) + \frac{1-\alpha_v}{{a^+}^\intercal v} \left(\delta^+ a^+  + \nabla f(v^+) \right)
\end{equation}
where
\begin{align*}
    \delta^- & = \left(\nabla f(v^*)  -  \nabla f(v^-) \right)^\intercal v^-  \leq 0 \\
    \delta^+ & = \left(\nabla f(v^*) -  \nabla f(v^+)\right)^\intercal v^+  \geq 0 
\end{align*}
for a vector $v^*$ contained on the segment $[v^-,v^+]$, and $\frac{1}{{a^-}^\intercal v} \coloneqq 0$ in the case $0\in P$.
\end{lemma}
\begin{proof}
Consider $v \in \mbox{int}(B)$ arbitrary. In this proof, to aid the reader, we emphasize that $v^\pm$ and $\alpha_v$ are functions of $v$ by referring to them as $v^\pm(v)$ and $\alpha_v(v)$, respectively.

Since $g(v)=\alpha_v(v) f(v^-(v)) + (1-\alpha_v(v)) f(v^+(v))$ and $f$ is differentiable, $g$ is also differentiable in the interior of $B$. The gradient of $g$ is given by 
\begin{align}
\nabla g(v)&=\nabla\alpha_v(v)\cdot f(v^-(v)) + \alpha_v(v) \nabla f(v^-(v)) \notag \\ 
& \qquad + \nabla (1-\alpha_v(v)) \cdot f(v^+(v)) + (1-\alpha_v(v)) \nabla f(v^+(v)) \notag \\
    &=\nabla\alpha_v(v)\cdot  \left( f(v^-(v)) - f(v^+(v))\right)  + \alpha_v(v) \nabla f(v^-(v)) \notag \\
    &\qquad + (1-\alpha_v(v)) \nabla f(v^+(v)) \label{eq:nablag}
\end{align}
where 
\begin{equation}
\nabla f(v^\pm(v))= {\mathcal{D} v^\pm(v)}^\intercal \nabla f(v)\rvert_{v=v^\pm}  \label{eq:nablaf}
\end{equation}
and $\mathcal{D}v^\pm(v)$ is the Jacobian matrix of  $v^\pm(v)$.
Recall that we are assuming $v^\pm(v)$ intersects a facet of $P$ contained in an hyperplane of equation ${a^\pm}^\intercal \vec{x} = 1$. Hence, by Remark~\ref{vpm} and defining $\frac{1}{{a^-}^\intercal v}\coloneqq 0$ when $0\in P$,
\begin{equation}
v^\pm(v)=\frac{1}{{a^\pm}^\intercal v}v, \ \text{ so } \ \nabla\left(\frac{1}{{a^\pm}^\intercal v}\right) = \frac{-a^\pm }{({a^\pm}^\intercal v)^2}
\end{equation}
and
\begin{align*}
\mathcal{D} v^\pm(v) &= v \nabla\left(\frac{1}{{a^\pm}^\intercal v}\right)^\intercal
+ \frac{1}{{a^\pm}^\intercal v}I_n\\
&=  \frac{-v {a^\pm}^\intercal }{({a^\pm}^\intercal v)^2}
+ \frac{1}{{a^\pm}^\intercal v}I_n
=  \frac{-1}{{a^\pm}^\intercal v} \left(v^\pm(v) a^{\pm \intercal} 
- I_n \right).
\end{align*}
Replacing in \eqref{eq:nablaf} we obtain 
\begin{align*}
\nabla f(v^\pm(v))& = \left(\frac{-1}{{a^\pm}^\intercal v} \left(v^\pm(v) {a^{\pm}}^ \intercal - I_n \right)\right)^\intercal \nabla f(v^\pm(v)) \\
& = \frac{-1}{{a^\pm}^\intercal v} \left((\nabla f(v^\pm(v))^\intercal v^\pm(v)) a^\pm 
- \nabla f(v^\pm(v)) \right)
\end{align*}

Note that ${\nabla v^\pm(v)}^\intercal v =0$ and $\nabla f(v^\pm(v))^\intercal v =0$. This is expected because varying $v$ over its ray does not change the position of $v^\pm(v)$ nor the value of $f(v^\pm(v))$.
On the other hand, 
applying the gradient to \eqref{eq:relAlphaHiperplanes} we obtain
\begin{align} \label{eq:nablaAlpha}
  \nabla \alpha_v(v) & = \left(\frac{1}{{a^-}^\intercal v}-\frac{1}{{a^+}^\intercal v}\right)^{-1} \left(\alpha_v(v)\frac{a^-}{({a^-}^\intercal v)^2}  + (1-\alpha_v(v))  \frac{a^+}{({a^+}^\intercal v)^2}\right)
 \end{align}

Finally, the mean value theorem ensures there exist  $v^*$ on the segment $[v^-(v),v^+(v)]$ such that
\begin{equation}  \label{eq:nablaf*}
 f(v^-(v)) - f(v^+(v)) = \nabla f(v^*)^\intercal (v^-(v) - v^+(v)) = \nabla f(v^*)^\intercal v \left(\frac{1}{{a^-}^\intercal v}-\frac{1}{{a^+}^\intercal v}\right)  
\end{equation}
Grouping all the terms  into \eqref{eq:nablag}, we obtain an explicit formula for the gradient of $g$ at $v$ as
\begin{align*}
\nabla g(v)&=\nabla\alpha_v(v)\cdot  \left( f(v^-(v)) - f(v^+(v))\right)  + \alpha_v(v) \nabla f(v^-(v))  + (1-\alpha_v(v)) \nabla f(v^+(v)) \\
&= \left(\alpha_v(v)\frac{a^-}{({a^-}^\intercal v)^2}  + (1-\alpha_v(v))  \frac{a^+}{({a^+}^\intercal v)^2}\right) \nabla f(v^*)^\intercal v\\
&\qquad + \alpha_v(v) \left( \frac{-1}{{a^-}^\intercal v} \left( \left( \nabla f(v^-(v))^\intercal v^-(v) \right) a^- -  \nabla f(v^-(v)) \right)\right)\\
&\qquad + (1-\alpha_v(v)) \left(\frac{-1}{{a^+}^\intercal v} \left( \left( \nabla f(v^+(v))^\intercal v^+(v) \right) a^+ -  \nabla f(v^+(v)) \right)\right)\\
&= \frac{\alpha_v(v)}{{a^-}^\intercal v}\left( \left(\left(\nabla f(v^*)^\intercal v \right) \frac{1}{{a^-}^\intercal v} -\left(  \nabla f(v^-(v))^\intercal v^-(v) \right) \right) a^- +   \nabla f(v^-(v)) \right) \\
&\qquad + \frac{1-\alpha_v(v)}{{a^+}^\intercal v} \left( \left(\left(\nabla f(v^*)^\intercal v \right) \frac{1}{{a^+}^\intercal v} -\left(  \nabla f(v^+(v))^\intercal v^+(v) \right) \right) a^+ +  \nabla f(v^+(v)) \right)\\
&= \frac{\alpha_v(v)}{{a^-}^\intercal v} \left( \underbrace{\left(\left(\nabla f(v^*)  -  \nabla f(v^-(v)) \right)^\intercal v^-(v) \right)}_{\delta^-}  a^- +   \nabla f(v^-(v)) \right) \\
&\qquad +  \frac{1-\alpha_v(v)}{{a^+}^\intercal v} \left( \underbrace{\left(\left(\nabla f(v^*) -  \nabla f(v^+(v))\right)^\intercal v^+(v)  \right)}_{\delta^+} a^+ +   \nabla f(v^+(v)) \right)\\
&= \frac{\alpha_v(v)}{{a^-}^\intercal v} \left(\delta^- a^-  + \nabla f(v^-(v)) \right) + \frac{1-\alpha_v(v)}{{a^+}^\intercal v} \left(\delta^+ a^+  + \nabla f(v^+(v)) \right)
\end{align*}

Note that since $f$ is ray-concave, then it is concave on the segment $[v^-(v),v^+(v)]$, so for any $v'$ on the segment
\begin{align*}
\left(\nabla f(v^*) - \nabla f(v^-(v))\right)^\intercal v' &\leq 0\\
\left(\nabla f(v^*) - \nabla f(v^+(v))\right)^\intercal v' &\geq 0
\end{align*}
because $\vec{0}, v, v^+(v)$ and $v^-(v)$ are colinear. Therefore, $\delta^+\geq 0$ and $\delta^-\leq 0$.
\end{proof}

As a side note, in the last lemma we only used differentiability of $f$ to show the formula \eqref{eq:gradientinside}, therefore such formula is always valid for $g$ defined as \eqref{eq:def}. Ray-concavity of $f$ was only used to show the signs of $\delta^\pm$, and facet-convexity of $f$ was not needed.

\subsection{Convexity over the polytope}

We know provide the last step which proves that $g$ is convex in $P$.
%

\begin{lemma}
Let $g: P\subset \mathbb{R}^n_+\rightarrow \mathbb{R}$ as defined in \eqref{eq:def}.
%
%
If $g$ is convex over each region $B\in \mathcal{B}$, then it is convex in $P$.
\end{lemma}
\begin{proof}
Our strategy to show convexity is to show \emph{mid-point local convexity}, that is, for each $v\in P$, we show there is a neighborhood of $v$ where $g$ is mid-point convex. We remind the reader that mid-point convexity reads
\[g\left(\frac{1}{2}(v_1 + v_2)\right)\leq \frac{1}{2}\left(g(v_1)+g(v_2)\right) \qquad \forall v_1,v_2\in P.\]
Mid-point convexity does not always imply convexity, but in this case it suffices as the function $g$ is continuous. Therefore, establishing local mid-point convexity implies local convexity \cite{jensen1905om}. And since local convexity implies convexity (see e.g. \cite{li2010some}), we conclude that $g$ is convex.

We now proceed to proving local mid-point convexity of $g$. Let us consider $v\in P$, $d\in \mathbb{R}^n$ and $\varepsilon > 0$. We would like to show that
\begin{equation}
    \label{eq:midpoint}
    g(v) \leq \frac{1}{2}(g(v-\varepsilon d) + g(v+\varepsilon d) ). 
\end{equation}
If $v \pm \varepsilon d \in \mbox{int}(B)$ the inequality follows from convexity of $g$ within a region. Therefore, we may assume $v\in B_s \cap B_t$, $v-\varepsilon d\in B_s$ and $v+\varepsilon d\in B_t$ for some $B_s,B_t \in \mathcal{B}$.



Let ${a_s^+}^\intercal x = 1$ be the \outhyperplane{} of $B_s$, and ${a_s^-}^\intercal x = 1$ be its \inhyperplane{}. Similarly, we define $a^\pm_t$. Thus, $v^\pm = \frac{1}{{a_s^\pm}^\intercal v} v = \frac{1}{{a_t^\pm}^\intercal v} v$. Let $\nabla g_{B_s}$ and $\nabla g_{B_t}$ be the gradients of $g$ in $\mbox{int}(B_s)$ and $\mbox{int}(B_t)$ respectively (see Lemma \ref{lemma:gradient}). Since these gradients are continuous, we can extend their formula \eqref{eq:gradientinside} to $B_s$ and $B_t$. From here, we obtain
\begin{align}
\nabla (g_{B_s}(v)-g_{B_t}(v)) &= \frac{\alpha_v}{{a_s^-}^\intercal v} \delta^- (a_s^- - a_t^-) + \frac{1-\alpha_v}{{a_s^+}^\intercal v} \delta^+ (a_s^+ - a_t^+)  
\end{align}
%
%
Now we focus on showing that $\nabla g_{B_s}(v)^\intercal d \leq \nabla g_{B_t}(v)^\intercal d$. Since $v+\varepsilon d \in B_t$,
\[ (v+\varepsilon d )^\pm = \frac{v+\varepsilon d}{{a_t^\pm}^\intercal (v+\varepsilon d)}.\]

We start exploring the facet contained in ${a_t^+}^\intercal x=1$. By convexity of the polytope $P$,  we know that ${a_s^+}^\intercal (v+\varepsilon d )^+ \leq 1$. Hence, 
${a_s^+}^\intercal (v+\varepsilon d ) \leq {a_t^+}^\intercal (v+\varepsilon d)$
and since ${a_s^+}^\intercal v = {a_t^+}^\intercal v$  we conclude that
\[ (a_s^+ - a_t^+)^\intercal d \leq 0. \]
In a similar way, for the facet contained in ${a_t^-}^\intercal x=1$, by convexity of the polytope we get that ${a_s^-}^\intercal (v+\varepsilon d )^- \geq 1$. So, ${a_s^-}^\intercal (v+\varepsilon d ) \geq {a_t^-}^\intercal (v+\varepsilon d)$ and we conclude that 
\[ (a_s^- - a_t^-)^\intercal d \geq 0. \]
%
As $\delta^-\leq 0$ and $\delta^+\geq 0$ (Lemma \ref{lemma:gradient}), we obtain that
\[ \nabla (g_{B_s}(v)-g_{B_t}(v))^\intercal d = \frac{\alpha_v}{{a_s^-}^\intercal v} \underbrace{\delta^- (a_s^- - a_t^-)^\intercal d}_{\leq 0} + \frac{1-\alpha_v}{{a_s^+}^\intercal v} \underbrace{\delta^+ (a_s^+ - a_t^+)^\intercal d}_{\leq 0} \leq 0.\]
so we conclude that
\[ \nabla g_{B_s}(v)^\intercal d \leq \nabla g_{B_t}(v)^\intercal d  \]
%


%
Finally, we can use the first order characterization of convexity within each region and obtain
\begin{align*}
  g(v+\varepsilon d) = g_{B_t}(v+\varepsilon d) \geq g(v) + \varepsilon \nabla g_{B_t}(v)^\intercal d \geq g(v) + \varepsilon \nabla g_{B_s}(v)^\intercal d\\
  g(v-\varepsilon d) = g_{B_s}(v-\varepsilon d) \geq g(v) - \varepsilon \nabla g_{B_s}(v)^\intercal d \geq g(v) - \varepsilon \nabla g_{B_t}(v)^\intercal d
\end{align*}
These two inequalities imply \eqref{eq:midpoint}. This completes the proof of local mid-point convexity of $g$ which, as discussed at the beginning of this proof, implies convexity of $g$ in $P$.
\end{proof}

Note that, similarly to Lemma \ref{lemma:gradient}, the latter proof does not explicitly rely on facet-convexity. The result mainly uses that $g$ is convex on each region and that $f$ is ray-concave (in order to use the signs of $\delta^\pm$ in the gradient formula).

Knowing that $g$ defines a convex function over the domain, we can prove our main theorem, showing that it corresponds to the convex envelope of $f$ over the polytope $P$.

\begin{proof}[Theorem~\ref{thm:main}]
By previous lemma, we know that $g$ is a convex function over the domain $P$. We show that $g$ is an underestimator of $f$, that is, $g(v) \leq f(v)$ for all $v\in P$. For $v=0$ it clearly holds.
If $v\neq 0$, $v\in P$ implies that $\alpha_v \in [0,1]$. Additionally, since $f$ is concave over $[v^-,v^+]$ we know that 
\[ f(v) = f(\alpha_v v^- + (1-\alpha_v) v^+) \geq \alpha_v f(v^-) + (1-\alpha_v) f(v^+) = g(v).\]

Finally, we argue why $g$ is the largest convex function that underestimates $f$. Let $h$ another convex function that underestimates $f$ and let $v\in P$ such that $h(v)>g(v)$.
Restricted to the segment $[v^-,v^+]$, the function $h$ is also convex. But this is a contradiction, because $f$ is concave on  $[v^-,v^+]$, so the largest convex function underestimating $f$ on this segment is the line interpolating $f(v^-)$ and $f(v^+)$, which is exactly $g$. 
\end{proof}

\subsection{On the positively homogeneous condition}\label{sec:positivehom}

In this section we present characterizations for when the function $g$ constructed in \eqref{eq:def} is positively homogeneous.
\begin{lemma}
%
If $0\in P$, then $g$ is positively homogeneous if and only if $f(0)=0$. In this case, 
\[ g(v) = {a^+}^\intercal v \cdot   f(v^+),\]
where ${a^{+}}^\intercal x = 1$ is the \outhyperplane{} of the region $B\ni v$ (see Remark \ref{vpm}).
\end{lemma}
\begin{proof}
If $0\in P$ then $v = \alpha_v \cdot 0 + (1-\alpha_v) v^+ = \frac{1-\alpha_v}{{a^+}^\intercal v} v$, so 
\[ g(v) = (1-{a^+}^\intercal v) f(0) + {a^+}^\intercal v \cdot f(v^+)\]
If $g$ is positively homogeneous, then $g(0) = 0 = f(0)$. To prove the other direction, if $f(0)=0$ then $g(v) = {a^+}^\intercal v \cdot   f(v^+)$, which is homogeneous because for any $\lambda>0$ such that $\lambda v\in P$, $(\lambda v)^+ = v^+$ so $g(\lambda v) = {a^+}^\intercal (\lambda v) \cdot f(v^+) = \lambda g(v)$. 
\end{proof}
As mentioned in Remark~\ref{remark:first}, the condition $f(0) = 0$ is not restrictive in the construction of convex envelopes when $0\in P$. If $f(0) \neq 0$, it suffices to define $\hat{f} = f - f(0)$ and use our construction to derive $\vex{\hat{f}}$. The desired convex envelope simply follows from noting that $\vex{f} = \vex{\hat{f}}+f(0)$. We illustrate the use of this transformation in the upcoming examples section.
\begin{lemma}
%
If $0 \notin P$, then $g$ is positively homogeneous iff, for every $v\in P$, ${{a^-}^\intercal v} \cdot f(v^-) = {{a^+}^\intercal v} \cdot f(v^+)$, where ${a^{\pm}}^\intercal x = 1$ are the \inhyperplane{} and \outhyperplane{} of a region $B\ni v$ (see Remark \ref{vpm}). In this case,
\[ g(v) = {{a^-}^\intercal v} \cdot f(v^-) = {{a^+}^\intercal v} \cdot f(v^+). \]
\end{lemma}
\begin{proof}
Since $v^\pm = \frac{1}{{a^\pm}^\intercal v} v$, if $g$ is homogeneous then 
 \[ g(v) = g\left( ({a^\pm}^\intercal v)\cdot  v^\pm\right) = {a^\pm}^\intercal v \cdot  g(v^\pm) = {{a^\pm}^\intercal v}\cdot  f(v^\pm).\]
For the other direction, if ${{a^-}^\intercal v} \cdot f(v^-) = {{a^+}^\intercal v} \cdot f(v^+)$, by \eqref{eq:relAlphaHiperplanes} we obtain
\begin{align*}
    g(v) &= \alpha_v f(v^-) + (1-\alpha_v) f(v^+) \\
    &= \alpha_v f(v^-) + \left( 1 - \frac{\alpha_v}{{a^-}^\intercal v} \right) ({a^+}^\intercal v) f(v^+) \\
    &= \alpha_v f(v^-) +  \left( 1 - \frac{\alpha_v}{{a^-}^\intercal v} \right) ({a^-}^\intercal v) f(v^-)\\
    &= \alpha_v f(v^-) +  ({a^-}^\intercal v) f(v^-) - \alpha_v f(v^-)\\
    &=  ({a^-}^\intercal v) f(v^-) =  ({a^+}^\intercal v) f(v^+)
\end{align*}
So, $g$ is homogeneous because for any $\lambda>0$ such that $\lambda v\in P$,  $g(\lambda v) = {a^\pm}^\intercal (\lambda v) \cdot f((\lambda v)^\pm) = \lambda ({a^\pm}^\intercal v) \cdot f(v^\pm) = \lambda g(v)$.
\end{proof}
We note that our results have an unexpected consequence: when $f$ is a homogeneous function, convexity of $f$ over the facets of $P$ imply convexity of $f$ over all $P$.
\begin{corollary}\label{cor:f_homogeneous}
Let $f:P\rightarrow \mathbb{R}$ is continuously differentiable and convex (concave) over the facets of $P$. If $f$ is positively homogeneous, then $f$ is convex (concave) over $P$.
\end{corollary}
\begin{proof}
We show the proof for $f$ convex on the facets; the concave case is almost identical. 
If $f$ is positively homogeneous then in particular is ray-linear. Hence $f(v) = ({a^+}^\intercal v ) \cdot f(v^+) = g(v)$. In addition, since $f$ is convex on the facets of $P$, by Theorem~\ref{thm:main} $ g = \vex{f}$, so $f$ is convex over $P$.
\end{proof}

\section{Examples of ray-concave functions and their envelopes}


In this section, we provide the convex envelopes of various explicit functions. Some of these are new, and some have been provided in the literature before. In the latter case, our result provides new perspectives, and in some cases simpler derivations.

\begin{example} Consider the function $f(x,y)= -x\cdot y$, whose convex envelope over $[l_x,u_x]\times[l_y,u_y]$ is well-known. In order to construct its convex envelope using Theorem \ref{thm:main}, we first shift the domain by considering the function 
\[\hat{f}(x,y)= f(x+l_x,y+l_y) + l_x\cdot l_y = -(x+l_x)\cdot (y+l_y) + l_x\cdot l_y\]
over the box $[0, u_x-l_x]\times[0,u_y-l_y]$.
It is easy to verify that $\hat{f}$ is ray-concave and linear on the facet of any box $[0, u_x-l_x]\times[0,u_y-l_y]$. Theorem \ref{thm:main} implies that $\vex{\hat{f}}(v) = {a^+}^\intercal v \cdot \hat{f}(v^+)$ and thus 
\begin{align*}
    \vex{\hat{f}}(x,y)
    &= \begin{cases}
    \frac{y}{u_y - l_y} \cdot \hat{f}\left(x\cdot\frac{u_y - l_y}{y},y\cdot\frac{u_y - l_y}{y}\right)& \text{if } y \geq \frac{u_y-l_y}{u_x- l_x} x \\
    \frac{x}{u_x - l_x} \cdot \hat{f}\left(x\cdot\frac{u_x - l_x}{x},y\cdot\frac{u_x - l_x}{x}\right)  & \text{if } y \leq \frac{u_y-l_y}{u_x- l_x} x \\
    \end{cases} \\
    &=\begin{cases}
    -u_y x - l_x y & \text{if } y \geq \frac{u_y-l_y}{u_x- l_x} x \\
    -l_y x - u_x y & \text{if } y \leq \frac{u_y-l_y}{u_x- l_x} x 
    \end{cases}
\end{align*}
And since $\vex{f}(x,y) = \vex{\hat{f}}(x-l_x,y-l_y)-l_x\cdot l_y$, we obtain
\[ \vex{f}(x,y) = \begin{cases}
-u_y x - l_x y + l_x u_y & \text{if } y-l_y \geq \frac{u_y-l_y}{u_x- l_x} (x-l_x) \\
-l_y x - u_x y + l_y u_x & \text{if } y-l_y \leq \frac{u_y-l_y}{u_x- l_x} (x-l_x) 
\end{cases} \]
which corresponds to the McCormick envelopes for this function.
\end{example}

\begin{example} Let us consider the following example from \cite{locatelli2014convex}.
Let $f(x,y)=y/x$ and 
\[ P = \left\{ (x,y)\in \mathbb{R}^2 : -x + 2y \leq 2, 1 \leq x \leq 2, 0 \leq y \leq 2 \right\} \]
We shift the domain by considering the function as $\hat{f}(x,y) = f(x+1,y)$ and the polytope $\hat{P}=\{(x,y)\in \mathbb{R}^2 : (x+1,y)\in P\}$. 
Note that $\hat{f}(x,y)$ is ray-concave because $\hat{f}(x,\lambda x)=\lambda \frac{ x}{x+1}$ is concave for $x\geq 0$ and $\lambda\geq 0$. Convexity on the facets can be directly verified.
Applying Theorem~\ref{thm:main}, since the outer facets of $\hat{P}$ are $x=1$ and $-\tfrac{1}{3} x + \tfrac{2}{3} y = 1$, we obtain
\begin{align*}
    \vex{\hat{f}}(x,y) &=
    \begin{cases}
    (-\tfrac{1}{3} x + \tfrac{2}{3} y) \hat{f}\left(\frac{x}{-\tfrac{1}{3} x + \tfrac{2}{3} y}, \frac{y}{-\tfrac{1}{3} x + \tfrac{2}{3} y}\right) & \text{if } y \geq 2x \\
    x \hat{f}\left(\frac{x}{x}, \frac{y}{x}\right) & \text{if } y \leq 2x \\
    \end{cases} \\
    &=
    \begin{cases}
    y \frac{-\tfrac{1}{3} x + \tfrac{2}{3} y}{x+\left(-\tfrac{1}{3} x + \tfrac{2}{3} y\right)}  = y \frac{-x + 2 y}{2x+2 y} & \text{if } y \geq 2x \\
    \frac{1}{2} y & \text{if } y \leq 2x \\
    \end{cases} 
\end{align*}
Therefore, as $\vex{f}(x,y) = \vex{\hat{f}}(x-1,y)$ we obtain
\[
    \vex{f}(x,y) =    \begin{cases}
     y \frac{1-x + 2 y}{2(x+ y + 1)} & \text{if } y \geq 2(x-1) \\
    \frac{1}{2} y & \text{if } y \leq 2(x-1) \\
    \end{cases} 
\]
\end{example}

\begin{example}
Let us consider the function
\begin{equation}
     \label{eq:probfunction2D}
     f(x,y) = \frac{xy}{x+y-xy}
 \end{equation}
in a box $[0,u_x]\times [0,u_y]\subseteq [0,1]^2$. This function appears naturally in the context of network reliability optimization. In fact, if $X,Y$ are independent Bernoulli random variables indicating the current state of two serial component, with reliabilities $p_X:=\mathbb{P}(X=1)$ and $p_Y:=\mathbb{P}(Y=1)$ then 
 \[f(p_X,p_Y) = \mathbb{P}\left(X\cdot Y = 1 | X+Y \geq 1\right).\]
corresponds to the resulting reliability of a \emph{degree-2} reduction~\cite{satyanarayana1985linear}.

We compute the concave envelope of \eqref{eq:probfunction2D} via the convex envelope of $\hat{f} = - f$. 
The function $\hat{f}$ can be directly verified to be convex on the facets of $[0,u_x]\times [0,u_y]$. For instance
\[h(x) = \hat{f}(x,u_y) = -\frac{xu_y}{x+u_y-xu_y}, \]
and a simple calculation shows
\[h''(x) = \frac{2(1-u_y)u_y^2}{(1-(1-u_y)(1-x))^3} \geq 0.\]
As for ray-concavity, we compute
\[
\hat{f}(x,\lambda x)  = -\frac{\lambda x}{1+\lambda - \lambda x} \quad \Rightarrow \quad 
\frac{\partial^2 \hat{f}}{\partial x^2}(x,\lambda x) = -\frac{2 \lambda^2 (1+\lambda)}{(1+\lambda(1-x))^3},
\]
therefore $\frac{\partial^2 f}{\partial x^2}(x,\lambda x) \geq 0$ for $x\leq 1$ and $\lambda \geq 0$. 
By Theorem \ref{thm:main}, the concave envelope of $f(x,y)$, denoted $\ave{f}$, is given by
\begin{align*}
    \ave{f}(x,y) = -\vex{\hat{f}}(x,y) &=  -({a^+}^\intercal v)  \hat{f}(v^+_x,v^+_y) \\
    &= \begin{cases}
    \frac{y}{u_y} \frac{x\frac{u_y}{y}\cdot u_y}{x\frac{u_y}{y} + u_y - x\frac{u_y}{y}\cdot u_y}   & \text{if }y \geq \frac{u_y}{u_x} x \\
    \frac{x}{u_x} \frac{u_x\cdot y\frac{u_x}{x}}{u_x + y\frac{u_x}{x} - u_x\cdot y\frac{u_x}{x}}   & \text{if }y \leq \frac{u_y}{u_x} x \\
    \end{cases} \\
    &=\begin{cases}
    \frac{x\cdot y}{x+y-x\cdot u_y}  &\text{if }y \geq \frac{u_y}{u_x} x \\
    \frac{x\cdot y}{x+y-u_x\cdot y}  & \text{if }y\leq \frac{u_y}{u_x} x \\
    \end{cases}
\end{align*}

Note that this procedure also computes, for free, the concave envelope of $f$ on the non-rectangular polytopes $\{(x,y)\in [0,u_x]\times [0,u_y]\, :\, y\leq \frac{u_y}{u_x} x\}$ and $\{(x,y)\in [0,u_x]\times [0,u_y]\, :\, y\geq \frac{u_y}{u_x} x\}$.
\end{example}

\begin{example} Let us consider the function
\[f(x,y)=-\frac{y \left(x^3 y^2+2 x^4 y-3 x^3 y-x^2 y+x^5-3 x^4+2 x^3-2 x y^2-y^3\right)}{x (x+y)^2} \]
over the region  $P = \left\{ (x,y)\in \mathbb{R}^2_+ : 1 \leq x+y \leq 2 \right\}$. 

Note that $f(x,0)=0$, $\lim_{\epsilon\to 0}f(\epsilon,y)$ is  linear, $f(x,1-x) = (x-1)^2/x$ and $f(x,2-x) = (x-2)^2/x$, so $f$ is convex on the facets. On the other hand, over the ray $y=\lambda x$ for $\lambda \geq 0$ we obtain
\[ \frac{\partial^2 f}{\partial x^2} (x,\lambda x) = -6\lambda \frac{(1+\lambda)x - 1}{1+\lambda}\] 
If $(x,\lambda x)\in P$ then $x\in [\tfrac{1}{1+\lambda}, \tfrac{2}{1+\lambda}]$, so $f$ is ray-concave for any $\lambda\geq 0$.  Applying  Theorem~\ref{thm:main}, we obtain:
\begin{align*}
    \vex{f}(x,y) &= (2-x-y)\cdot f\left(\frac{x}{x+y},\frac{y}{x+y}\right) + (x+y-1) \cdot  f\left(\frac{2x}{x+y},\frac{2y}{x+y}\right) \\
    &= (2-x-y)\cdot \frac{y^2}{x (x + y)}+ (x+y-1) \cdot\frac{2y^2}{x (x + y)} = \frac{y^2}{x}
\end{align*}
which corresponds to the convex envelope because ${y^2}/{x}$ is positively homogeneous.
\end{example}

The following example shows how Corollary~\ref{cor:f_homogeneous} can be used to prove the convexity of positively homogeneous functions.

\begin{example}
Let $f$ be a 3-dimensional Cobb-Douglas function \[f(x_1,x_2,x_3) = A x_1^{\alpha_1}x_2^{\alpha_2}x_3^{\alpha_3}\]
where $A,\alpha_1,\alpha_2,\alpha_3>0$, $\alpha_1+\alpha_2+\alpha_3= 1$ and $\vec{x}\in\mathbb{R}^3_+$. 

It is known that the 2-dimensional Cobb-Douglas function is concave if $\alpha_i+\alpha_j< 1$, hence  $f$ is concave over the facets of the box $P=[l_1,u_1]\times[l_2,u_2]\times[l_3,u_3] \subset \mathbb{R}^3_+$.

Since $\sum_{i=1}^3 \alpha_i= 1$, $f$ is positively homogeneous, so by Corollary~\ref{cor:f_homogeneous} we conclude that $f$ is concave over $P$.
\end{example}

\bibliographystyle{amsplain}      
\bibliography{convexEnvelopes}

\providecommand{\bysame}{\leavevmode\hbox to3em{\hrulefill}\thinspace}
\providecommand{\MR}{\relax\ifhmode\unskip\space\fi MR }
\providecommand{\MRhref}[2]{%
  \href{http://www.ams.org/mathscinet-getitem?mr=#1}{#2}
}
\providecommand{\href}[2]{#2}
\begin{thebibliography}{10}

\bibitem{al1983jointly}
Faiz~A Al-Khayyal and James~E Falk, \emph{Jointly constrained biconvex
  programming}, Mathematics of Operations Research \textbf{8} (1983), no.~2,
  273--286.

\bibitem{anstreicher2010computable}
Kurt~M Anstreicher and Samuel Burer, \emph{Computable representations for
  convex hulls of low-dimensional quadratic forms}, Mathematical Programming
  \textbf{124} (2010), no.~1, 33--43.

\bibitem{bao2009multiterm}
Xiaowei Bao, Nikolaos~V Sahinidis, and Mohit Tawarmalani, \emph{Multiterm
  polyhedral relaxations for nonconvex, quadratically constrained quadratic
  programs}, Optimization Methods \& Software \textbf{24} (2009), no.~4-5,
  485--504.

\bibitem{benson2004construction}
Harold~P Benson, \emph{On the construction of convex and concave envelope
  formulas for bilinear and fractional functions on quadrilaterals},
  Computational Optimization and Applications \textbf{27} (2004), no.~1, 5--22.

\bibitem{jach2008convex}
Matthias Jach, Dennis Michaels, and Robert Weismantel, \emph{The convex
  envelope of (n--1)-convex functions}, SIAM Journal on Optimization
  \textbf{19} (2008), no.~3, 1451--1466.

\bibitem{jensen1905om}
Johan Ludwig William~Valdemar Jensen, \emph{Om konvekse funktioner og uligheder
  imellem middelvaerdier}, Nyt tidsskrift for matematik \textbf{16} (1905),
  49--68.

\bibitem{khajavirad2012convex}
Aida Khajavirad and Nikolaos~V Sahinidis, \emph{Convex envelopes of products of
  convex and component-wise concave functions}, Journal of Global Optimization
  \textbf{52} (2012), no.~3, 391--409.

\bibitem{khajavirad2013convex}
\bysame, \emph{Convex envelopes generated from finitely many compact convex
  sets}, Mathematical Programming \textbf{137} (2013), no.~1, 371--408.

\bibitem{kuno2002branch}
Takahito Kuno, \emph{A branch-and-bound algorithm for maximizing the sum of
  several linear ratios}, Journal of Global Optimization \textbf{22} (2002),
  no.~1, 155--174.

\bibitem{li2010some}
Yuan-Chuan Li and Cheh-Chih Yeh, \emph{Some characterizations of convex
  functions}, Computers \& Mathematics with applications \textbf{59} (2010),
  no.~1, 327--337.

\bibitem{liberti2003convex}
Leo Liberti and Constantinos~C Pantelides, \emph{Convex envelopes of monomials
  of odd degree}, Journal of Global Optimization \textbf{25} (2003), no.~2,
  157--168.

\bibitem{linderoth2005simplicial}
Jeff Linderoth, \emph{A simplicial branch-and-bound algorithm for solving
  quadratically constrained quadratic programs}, Mathematical Programming
  \textbf{103} (2005), no.~2, 251--282.

\bibitem{locatelli2016polyhedral}
Marco Locatelli, \emph{Polyhedral subdivisions and functional forms for the
  convex envelopes of bilinear, fractional and other bivariate functions over
  general polytopes}, Journal of Global Optimization \textbf{66} (2016), no.~4,
  629--668.

\bibitem{locatelli2018convex}
\bysame, \emph{Convex envelopes of bivariate functions through the solution of
  {KKT} systems}, Journal of Global Optimization \textbf{72} (2018), no.~2,
  277--303.

\bibitem{locatelli2020convex}
\bysame, \emph{Convex envelope of bivariate cubic functions over rectangular
  regions}, Journal of Global Optimization \textbf{76} (2020), no.~1, 1--24.

\bibitem{locatelli2014convex}
Marco Locatelli and Fabio Schoen, \emph{On convex envelopes for bivariate
  functions over polytopes}, Mathematical Programming \textbf{144} (2014),
  no.~1, 65--91.

\bibitem{luedtke2012some}
James Luedtke, Mahdi Namazifar, and Jeff Linderoth, \emph{Some results on the
  strength of relaxations of multilinear functions}, Mathematical Programming
  \textbf{136} (2012), no.~2, 325--351.

\bibitem{mccormick1976computability}
Garth~P McCormick, \emph{Computability of global solutions to factorable
  nonconvex programs: Part {I} -- {C}onvex underestimating problems},
  Mathematical Programming \textbf{10} (1976), no.~1, 147--175.

\bibitem{meyer2004trilinear}
Clifford~A Meyer and Christodoulos~A Floudas, \emph{Trilinear monomials with
  mixed sign domains: Facets of the convex and concave envelopes}, Journal of
  Global Optimization \textbf{29} (2004), no.~2, 125--155.

\bibitem{meyer2005convex}
\bysame, \emph{Convex envelopes for edge-concave functions}, Mathematical
  Programming \textbf{103} (2005), no.~2, 207--224.

\bibitem{muller2020using}
Benjamin Muller, Felipe Serrano, and Ambros Gleixner, \emph{Using
  two-dimensional projections for stronger separation and propagation of
  bilinear terms}, SIAM Journal on Optimization \textbf{30} (2020), no.~2,
  1339--1365.

\bibitem{rikun1997convex}
Anatoliy~D Rikun, \emph{A convex envelope formula for multilinear functions},
  Journal of Global Optimization \textbf{10} (1997), no.~4, 425--437.

\bibitem{ryoo2001analysis}
Hong~Seo Ryoo and Nikolaos~V Sahinidis, \emph{Analysis of bounds for
  multilinear functions}, Journal of Global Optimization \textbf{19} (2001),
  no.~4, 403--424.

\bibitem{satyanarayana1985linear}
Appajosyula Satyanarayana and R~Kevin Wood, \emph{A linear-time algorithm for
  computing k-terminal reliability in series-parallel networks}, SIAM Journal
  on Computing \textbf{14} (1985), no.~4, 818--832.

\bibitem{sherali1997convex}
Hanif~D Sherali, \emph{Convex envelopes of multilinear functions over a unit
  hypercube and over special discrete sets}, Acta Mathematica Vietnamica
  \textbf{22} (1997), no.~1, 245--270.

\bibitem{sherali1990explicit}
Hanif~D Sherali and Amine Alameddine, \emph{An explicit characterization of the
  convex envelope of a bivariate bilinear function over special polytopes},
  Annals of Operations Research \textbf{25} (1990), no.~1, 197--209.

\bibitem{tardella2004existence}
Fabio Tardella, \emph{On the existence of polyhedral convex envelopes},
  Frontiers in Global Optimization, Springer, 2004, pp.~563--573.

\bibitem{tardella2008existence}
\bysame, \emph{Existence and sum decomposition of vertex polyhedral convex
  envelopes}, Optimization Letters \textbf{2} (2008), no.~3, 363--375.

\bibitem{tawarmalani2013explicit}
Mohit Tawarmalani, Jean-Philippe~P Richard, and Chuanhui Xiong, \emph{Explicit
  convex and concave envelopes through polyhedral subdivisions}, Mathematical
  Programming \textbf{138} (2013), no.~1, 531--577.

\bibitem{tawarmalani2001semidefinite}
Mohit Tawarmalani and Nikolaos~V Sahinidis, \emph{Semidefinite relaxations of
  fractional programs via novel convexification techniques}, Journal of Global
  Optimization \textbf{20} (2001), no.~2, 133--154.

\bibitem{zamora1999branch}
Juan~M Zamora and Ignacio~E Grossmann, \emph{A branch and contract algorithm
  for problems with concave univariate, bilinear and linear fractional terms},
  Journal of Global Optimization \textbf{14} (1999), no.~3, 217--249.

\end{thebibliography}
\end{document}